\documentclass[a4paper,10pt]{article}
\usepackage[cp1251]{inputenc}
\usepackage{amssymb, amsmath, amsthm}
\usepackage[english]{babel}
\title{Interpolation of analytic functions of moderate growth in the unit disc and zeros of solutions of a linear differential equation}
\author{Igor Chyzhykov, Iryna Sheparovych}

\def\vfi{\varphi}
\def\Z{\Bbb Z}
\def\N{\Bbb N}

\newcommand{\D}{\mathbb{D}}
\newcommand{\Bl}{\Bigl(}
\newcommand{\Br}{\Bigr)}
\newcommand{\Bm}{\Bigl|}
\newcommand{\rr}[1]{re^{i{#1}}}

\newtheorem{theorem}{Theorem}
\newtheorem{prop}{Proposition}
\newtheorem{lem}{Lemma}
\newtheorem*{corollary*}{Corollary}
\newtheorem{cor}{Corollary}

\theoremstyle{remark}
\newtheorem{rem}{Remark}

\def\vfi{\varphi}

\def\Z{\mathbb Z}
\def\N{\mathbb N}

\def\vfi{\varphi}

\def\D{\mathbb D}

\renewcommand{\Re}{\mathop{\rm Re}}

\newcounter{zad}[subsection]
\begin{document}
\maketitle
\begin{abstract}
In 2002 A.\  Hartmann and X.\ Massaneda obtained necessary and sufficient conditions for  interpolation sequences for classes of analytic functions in the unit disc such that $\log M(r,f)=O((1-r)^{-\rho})$, $0<r<1$, $\rho \in (0 , +\infty)$, where $M(r,f)=\max\{ |f(z)|: |z|=r\}$. Using another method, we give an explicit construction of an interpolating function in this result.
As an application  we describe minimal growth of  the coefficient $a$  such that the equation $f''+a(z)f=0$ possesses a solution with a prescribed sequence of zeros.

MathSubjClass 2010: 30C15, 30H05, 30H99, 30J99.

Keywords: {analytic function, moderate growth, unit disc, interpolation, canonical product, growth}
\end{abstract}

\section{Introduction and  results}
\subsection{Interpolation in the unit disc} Let  $(z_n)$ be a sequence of different complex numbers in the unit disc $\mathbb{D}=\{ z: |z|<1\}$, and let $\sigma(z,
\zeta)=\Bigl| \frac{z-\zeta}{1-\bar z \zeta}\Bigr|$ denote the pseudohyperbolic distance in $\mathbb{D}$. Let $U(z,t)=\{\zeta\in \mathbb{C} : |\zeta-z|<t\}$.  In the sequel, the symbol $C$ stands for positive constants which depend on the parameters indicated, not necessarily the same at each occurrence. We say that the sequence
 $(z_n)$ is  \emph{uniformly discrete} or \emph{separated}, if $\inf _{j\ne k} \sigma(z_k,
z_j) >0$. L.~Carleson  (\cite{Carles58}, \cite{Dur}) consider
the problem of description of so-called \emph{universal interpolation sequences} or \emph{interpolation sets} for the class
 $H^\infty$ of bounded analytic functions in $\mathbb{D}$,
i.e. those sequences $(z_k)$ in $\mathbb{D}$ that  $\forall
(b_k)\in l^\infty$ there exists $f\in H^\infty$ with
\begin{equation}\label{e:inter_prob}
f(z_k)=b_k.
\end{equation}
He proved that $(z_k)$ is a universal interpolation sequence for $H^\infty$ if and only if
\begin{equation}\label{e:uniform_separ}
    \exists \delta>0: \quad \prod_{j\ne k}  \sigma(z_j, z_k)  \ge \delta, \quad k\in \mathbb{N}.
\end{equation}
For the similar problems in  $H^p$ see \cite[Chap. 9]{Dur}.

For the Banach space
$A^{-n}$, $n>0$, of analytic functions such that $\|f\|^\infty_n=\sup_{z\in \mathbb{D}} (1-|z|)^n
|f(z)|<\infty$, an interpolation set is defined by the condition that for every sequence $(b_k)$ with  $(b_k (1-|z_k|)^n) \in l^\infty$ there is a function $f\in A^{-n}$ satisfying~\eqref{e:inter_prob}. These sets were described by K.~Seip in  \cite{Seip93}. Namely, necessary and sufficient that $(z_k)$  be an interpolation set for $A^{-n}$ is that $(z_n)$ be separated and $\mathcal{D}^+(Z)<n$ where
\begin{equation}\label{e:Dplus}
    \mathcal{D}^+(Z)=\varlimsup_{r\uparrow  1} \sup_{z \in \mathbb{D}} \frac{\sum\limits_{\frac 12< \sigma(z, z_j)<r}\ln \frac 1{\sigma(z, z_j)}}{\ln \frac 1{1-r}}.
\end{equation}
We note that the condition  \eqref{e:uniform_separ} implies boundedness of the numerator
in~\eqref{e:Dplus}.

For an analytic function $f$ in  $\mathbb{D}$ we denote $M(r,f)=\max
\{|f(z)|: |z|=r\}$, $r\in (0,1)$. Let
$n_\zeta(t)=\sum_{|z_k-\zeta|\le t} 1 $ be the number of the members of the sequence  $(z_k)$ satisfying $|z_k-\zeta|\le t$.
We write  $$N_\zeta(r)=\int_0^r
\frac{(n_\zeta(t)-1)^+}{t}dt.$$

The  results mentioned above cannot be applied to analytic functions $f$ such that $\ln
\frac 1{1-r}=o(\ln M(r,f))\ (r\uparrow 1)$. In 1956 A.\ G.\  Naftalevich \cite{Naf} described interpolation sequences for the Nevanlinna class.
On the other hand,  a description of interpolation sets
in the class of analytic functions in the unit disc and of infinite order of the growth satisfying
$$ \exists C>0 \ \forall r\in(0;1):   \ln  \ln M(r,f) \le C \ln \gamma \Bigl(\frac C{1-r} \Bigr),$$
where $\ln\gamma (t)$ is a convex function in $\ln t$ and  $\ln
t=o(\ln\gamma(t))\ (t\to\infty)$, was found by B.~Vynnytskyi and I.~Sheparovych in 2001 (\cite{VSh01}).

Consider  the class of analytic functions such that
\begin{equation}\label{e:eta_type}
\exists C>0\ \forall r\in(0;1):     \ln M(r,f) \le C
\eta\Bigl(\frac C{1-r} \Bigr),
\end{equation}
where $\eta\colon [1, +\infty) \to (0,+\infty)$ is an increasing convex function in
 $\ln t$ such that  $\ln t=o(\eta(t))\
(t\to\infty)$. In 2001 in the PhD thesis of the second author \cite[Theorem 3.1]{Sh_PhD} (see also \cite{VSh04})
it was proved that given a sequence $(z_n)$ in $\mathbb{D}$, in order that for  every  $(b_n)$ such that  \begin{equation*}%\label{}
    \exists C>0 \ \forall n\in \mathbb{N} : \quad \log |b_n|\le C \eta \Bl \frac C{1-|z_n|}\Br
\end{equation*}
there exist an analytic function from the class \eqref{e:eta_type} satisfying \eqref{e:inter_prob},   it is necessary  that
\begin{equation}\label{e:concentr_cond_eta}
 \exists \delta\in (0,1)\  \exists C>0 \ \forall n\in \mathbb{N} :  \quad N_{z_n} (\delta (1-|z_n|)) \le \eta\Bl \frac C
{1-|z_n|}\Br.
 \end{equation}
In 2002 A.\ Hartmann  and X.\ Massaneda \cite{HarMas} proved that  condition \eqref{e:concentr_cond_eta} is actually necessary and sufficient for a class of
growth functions $\eta$ containing all power functions. They also describe interpolation sequences in the unit ball in $\mathbb{C}^n$ in the similar situation. Note that the proofs of necessity in both \cite{VSh04}, \cite{HarMas} used similar methods of functional analysis. On the other hand, the proof of sufficiency in \cite{HarMas} is based on $L^2$-estimate for the solution to a $\bar \partial$-equation and is non-constructive.

In 2007  A. Borichev, R.\ Dhuez and K.\ Kellay \cite{Borich_k} solved an interpolation problem in classes of functions of arbitrary growth in both the complex plane and the unit disc.  Following \cite{Borich_k} let $h\colon [0,1)\to [0,+\infty)$ such that  $h(0)=0$, $h(r) \uparrow \infty$ $(r \uparrow1)$.  Denote by $\mathcal{A}_h$  and $\mathcal{A}_h^p$, $p>0$ the Banach spaces  of analytic functions on $\D$ with the norms
$$ \| f\|_h=\sup_{z\in \D} |f(z)| e^{-h(z)}<+\infty, \quad  \| f\|_{h, p}= \biggl( \int_{\D} |f(z)| ^p e^{-p h(|z|)} dm_2(z)\biggr)^{\frac 1p}, $$
respectively.
We then suppose that  $h\in C^3([0,1))$, $\rho(r):=(\frac {d^2h(r)}{(d\log r)^2} )^{-\frac12}\searrow 0$, and $\rho'(r)\to 0$ as $r\uparrow1$, for all $K>0$:  $\rho(r+x)\sim \rho(r)$  for $|x|\le K\rho(r)$, $r\uparrow1$ provided that $K\rho(r)<1-r$, and either $\rho(r)(1-r)^{-c}$ increases for some finite $c$ or $\rho'(r) \ln \rho(r)\to 0$ as $r\uparrow1$. Note that these assumptions imply $h(r)/\ln \frac 1{1-r} \to+\infty$ $(r\uparrow1)$.

Given such an $h$ and a sequence $Z=(z_k)$ in $\D$ denote by $$\mathcal{D}_\rho^+(Z)=\limsup_{R\to\infty} \limsup_{|z|\uparrow1} \frac{\mathop{card}(Z\cap U(z, R\rho(z)))}{R^2}.$$

\begin{theorem}[Theorem 2.3 \cite{Borich_k}] A sequence $Z$ is an interpolating  set for $\mathcal{A}_h (\D)$ if and only if $\mathcal{D}_\rho^+(Z)< \frac 12$ and $\inf\limits _{k\ne n} \dfrac{|z_k-z_n|}{\min \{\rho(|z_k|), \rho(|z_n|)\}}>0$.
 \end{theorem}
The similar description holds for interpolation sets for the classes $\mathcal{A}_{h}^p (\D)$, $p>0$~(\cite{Borich_k}).

We give  an explicit construction of a function solving the interpolation problem  in the class %An aim of the paper is to give a more explicit description of interpolation sets for
$\mathcal{A}_h^\infty =\bigcup_{c<\infty} \mathcal{A}_{ch}$  when $h$  grows not faster than $(1-r)^{-\rho}$, $\rho>0$. Note that the restrictions on $h$ posed  in \cite[Section 5, Def.2]{HarMas}  do not allow growth smaller than that of  a power function. In particular, it does not admit   the choices $h(r)=(\log r)^\alpha$, $\alpha\ge 1$, $h(r)=\exp\{(\log r)^\beta\}$, $0<\beta< 1$. On the other hand, Theorem \ref{t:sufficient_cond_interpol} works in these cases.

In particular, we show (Theorem \ref{t:sufficient_cond_interpol}) that condition \eqref{e:concentr_cond_eta}  is also sufficient in the case when $\eta$ is a power function.

\subsection{Zeros of solutions of $f''+a(z)f=0$}
One of possible applications of the mentioned results is description of zero sequences of solutions of the differential equation
\begin{equation}\label{e:oscil_eq}
    f''+a(z)f=0,
\end{equation}
where $a(z)$ is an analytic function  in $\mathbb{D}$.

% see also \cite{Shavala}. We plan to consider applications in the next paper.
We deal with the following problem (cf.  \cite[Problem 2]{Heit_cmft}).

{\bf  Problem.} {\sl Let $(z_k)$ be a sequence of distinct points in $\D$ without limit points there. Find a function $a(z)$, analytic in $\D$ such that \eqref{e:oscil_eq} possesses a solution having zeros precisely at the points $z_k$. Estimate the growth of the resulting function  $a(z)$.}

In the case when $a$ is entire such investigation was initiated by V.\ \v Seda \cite{Seda} (we address the reader  to the paper \cite{HeitLain05} for further references).  A survey of results devoted to the case when $A$ is analytic in the unit disc, and  zeros of $f$ form a Blaschke  sequence, i.e. $\sum_k (1-|z_k|)<\infty$,  is given in \cite{Heitt_sur}. The case of an arbitrary domain is treated in \cite{Vyn_Shav12}.

In particular, in \cite{Shavala} it is shown that  one can always find a solution to this problem.
For an entire function $F(z)=\sum_{n=0}^\infty F_n z^n$ define the maximal term $\mu(r,F)=\max\{ |F_n|r^n: n\ge 0\}$.

\begin{theorem}[\cite{Shavala}]
 Let $(z_k)$ be  a Blaschke sequence of distinct points, $F$ be an entire function such that
\begin{equation}\label{e:max_term_cond}
 \mu\Bigl( \frac{1}{1-|z_k|},F\Bigr) \ge \frac 1{|B'(z_k)|^2},
\end{equation}
where $B$ is the Blaschke product constructed by $(z_k)$. Then there exists a function $a(z) $ analytic in $\D$ such that the equation \eqref{e:oscil_eq} possesses a solution $f$  whose zero sequence is $(z_k)$ and for some $c>0$
$$ M(r,f)\le \exp \Bigl( \frac {c}{(1-r)^5} \mu \Bigl(\frac{2}{1-r} , F\Bigr)\Bigr),\quad r\in [0,1).$$
\end{theorem}

Unfortunately, it is hard to check  \eqref{e:max_term_cond} for a given sequence.
On the other hand, additional restrictions on the sequence allow to obtain explicit growth estimates of the coefficient $a$  or a solution $f$. In particular, in \cite{Heit_cmft} it is proved that {\sl if $(z_k)$ is a Blaschke sequence satisfying \eqref{e:uniform_separ}, then there exists an $a(z)$ analytic in $\D$ such that \eqref{e:oscil_eq} possesses a bounded analytic solution with zero sequence $(z_k)$.}

Recently J. Gr\" ohn and J. Heittokangas have  also proved the following theorem.

\begin{theorem}[{\cite[Theorem 6(a)]{Gr_Heit}}]
Let $(z_k)$ be a non-zero sequence in $\D$ such that
\begin{equation*}
 \inf_{k\in \mathbb{N}} (1-|z_k|)^{-q} \prod_{j\ne k} \sigma(z_j, z_k)>0,
\end{equation*}
for some $q>0$, and $\sum_{k=1}^\infty (1-|z_k|)^\alpha<\infty$  for some $\alpha \in (0,1]$.
Then there exists a function $a\in A^{-2(1+\alpha+2q)}$ satisfying $\limsup_{|z|\uparrow1} |a(z)| (1-|z|^2)^2>1$ such that \eqref{e:oscil_eq} possesses a solution whose zero sequence  is $(z_k)$.
\end{theorem}

Our result (Theorem \ref{t:zeros_de}) complements those of Gr\"ohn and Heittokangas in the case when a zero sequence need not to satisfy the Blaschke condition, but still has a finite exponent of convergence. We also show (Theorem \ref{t:zeors_sharp}) that our estimate of the growth of the coefficient $a$ is sharp in some sense.

\subsection{Main results}
Let  $\psi\colon [1,+\infty)\to \mathbb{R}_+$ be a nondecreasing function. We define
$$ \tilde \psi(x)=\int_1^x \frac {\psi(t)}t dt.$$

Let, in addition, $\psi$ have  finite order in the sense of P\'olya, i.e.
\begin{equation}\label{e:polya}
\psi(2x)=O(\psi(x)), \quad x\to+\infty.
\end{equation}
\begin{rem}    Polya's order  $\rho^*[\psi]$   of $\psi$ (\cite{DrSh}) is characterized by the condition that for any $\rho>\rho^*[\psi]$, we have
\begin{equation}\label{e:psi_con}
 \psi(Cx)\le C^\rho \psi(x), \quad x, C\to \infty.
\end{equation}
Note that  \eqref{e:psi_con} implies
\begin{equation*}\label{e:til_psi_con}
\tilde \psi(Cx)\le C^\rho \tilde\psi(x) +\tilde\psi(C)\le    2C^\rho \tilde \psi(x), \quad x, C\to \infty.
\end{equation*}
so $\rho^*[\tilde \psi]\le \rho^*[\psi]$.

Also, it is known that \eqref{e:psi_con} implies that $\psi$ has finite order of growth, i.e. $\psi(x)=O(x^\rho)$, $x\to+\infty$, but not vise versa.
\end{rem}

\begin{rem} In the case $\psi(x)=\ln ^p x$, $p\ge 0$ we get $\tilde \psi(x)=\frac 1{p+1} \ln ^{p+1} x$, and in
the case $\psi(x)=x^\rho $, $\rho>0$ we have $\tilde \psi(x)=\frac 1{\rho} (x^{\rho} -1)$.
\end{rem}

The following theorem gives sufficient conditions for interpolation sequences in classes of analytic functions of moderate growth in the unit disc.

\begin{theorem} \label{t:sufficient_cond_interpol} Let  $(z_n)$ be a sequence of distinct complex numbers in $\mathbb{D}$.
Assume that for some nondecreasing unbounded function  $\psi\colon [1,+\infty)\to \mathbb{R}_+ $ satisfying \eqref{e:polya} we have \begin{equation}\label{e:concentr_cond_psi}
 \exists C>0:\    \ \forall n\in \mathbb{N}   \quad N_{z_n} \Bigl(\frac{1-|z_n|}2 \Bigr)\le C  \psi\Bl \frac 1{1-|z_n|}\Br
 \end{equation}
Then for any sequence  $(b_n)$ satisfying
 \begin{equation*}\label{e:b_n_con_psi}
 \exists C>0:\    { \ln |b_n|} \le C {\tilde \psi\Bl \frac {1}{1-|z_n|}}\Br , \quad n\in \mathbb{N}
     \end{equation*}
there exists an analytic function $f$ in  $\mathbb{D}$ with the property \eqref{e:inter_prob}
%\begin{equation}\label{e:interpol}
% f(z_n)=b_n
%\end{equation}
and
 \begin{equation}\label{e:growth_con_psi}
 \exists C>0:\ \ln M(r,f)\le  C \tilde\psi \Bl \frac {1}{1-r}\Br. \end{equation}
 \end{theorem}

In order to prove a criterion we define a class of `regularly' growing functions. The class $\mathcal{R}$ consists of functions $\psi \colon [1,+\infty)\to \mathbb{R}_+$ which are nondecreasing, and  such that
$\tilde \psi(r)=O(\psi(r))$ as $r\to+\infty$. We note that the power function $x^\rho$, $\rho>0$, belongs to $\mathcal{{R}}$. Also, given a positive nondecreasing function $\psi$, if for some $C>1$ there exists $t_0$ such that
$\psi(Ct)\ge 2 \psi (t)$ for all $t> t_0$, then $\psi \in \mathcal{{R}}$ (see \cite[p.50-51]{Shkal}).

Combining Theorem \ref{t:sufficient_cond_interpol} with the aforementioned result from \cite{VSh04} we are able to prove the following criterion, which essentially coincides with a result  from \cite{HarMas}.
%A function $\rho\colon[0,1) \to \mathbb{R}_+$ is called a \emph{proximate order} (\cite[p.55]{JK}, cf. \cite{Le}) if it satisfies the following conditions:
%\begin{itemize}
%  \item [(i)] $\rho$ is differentiable on $[0, 1)$;
%  \item [(ii)] $\lim_{r \uparrow 1} \rho(r)=\rho_0 \in [0,\infty)$;
%  \item [(iii)] $\lim_{r \uparrow 1} \rho'(r)(1-r)\log (1-r)=0$.
%\end{itemize}
%
%An advantage of this definition is that for every analytic function $f$ of finite positive order  there exists a proximate order $\rho(r)$ such that
%\[ \limsup_{r\uparrow 1} (1-r)^{\rho(r)}\log M(r,f)=1 .\]

\begin{theorem} \label{t:interpol_criter} Let  $(z_n)$ be a sequence of distinct complex number in  $\mathbb{D}$, and let $\psi \in \mathcal{R}$ satisfy \eqref{e:polya}.
The following  conditions are equivalent:
\begin{itemize}
  \item [(i)]  $\forall (b_n)$ such that  \begin{equation*}\label{e:b_n_con}
   \ \exists C>0: \   \ln |b_n| \le  {C} \psi\Bl \frac{1}{1-|z_n|}\Br, \quad n\in \mathbb{N}
     \end{equation*}
there exists an analytic function $f$ in $\mathbb{D}$ satisfying  \eqref{e:inter_prob}
%\begin{equation}\label{e:interpol}
% f(z_n)=b_n
%\end{equation}
and

 \begin{equation*}\label{e:growth_con}
\exists C>0:  \ln M(r,f)\le  C\psi \Bl \frac{1}{1-r}\Br; \end{equation*}

  \item [(ii)] condition \eqref{e:concentr_cond_psi} holds;
 % \begin{equation}\label{e:concentr_cond}
% \exists C_3>0  \ \forall n\in \mathbb{N}   \quad N_{z_n} \Bigl(\frac {1-|z_n|}2 \Bigr) \le  {C_3} \psi \Bl \frac{1}{1-|z_n|}\Br;
% \end{equation}
  \item [(iii)] \begin{equation*} \exists C>0  \ \forall n\in \mathbb{N} \sum\limits_{0< |z_n- z_j|<\frac12 (1-|z_n|)}\ln \frac 1{\sigma(z_n, z_j)}\le {C}\psi \Bl \frac{1}{1-|z_n|}\Br. \label{e:korenb_dens}\end{equation*}
\end{itemize}
 \end{theorem}

\begin{rem} \label{r:N_z}As it was proved in \cite{Sh00},  condition \eqref{e:concentr_cond_psi} is equivalent to
that $$ \exists\delta_1\in(0;1)\ \exists C\in(0;1)\ \forall z
\in \mathbb{D}:\  N_z(\delta_1(1-|z|))\leq C\psi\Bl\frac1{1-|z|}
\Br.
$$
\end{rem}

\begin{rem} Repeating the arguments from the proof of the equivalence $(ii)\Leftrightarrow (iii)$ on p.\pageref{p:2iff3}, one can prove that  \eqref{e:concentr_cond_psi} holds if and only if
 \begin{equation*} \exists C>0  \ \forall z\in \mathbb{D} : \sum\limits_{0< |z- z_j|<\frac12 (1-|z|)}\ln \frac 1{\sigma(z, z_j)}\le {C}\psi \Bl \frac{1}{1-|z|}\Br. %\label{e:korenb_dens}
 \end{equation*}
\end{rem}

%\begin{rem}
%Standard  properties of a proximate order imply that for $\psi(t)=t^{\rho(1-\frac 1t)}$, one has $\tilde\psi(t)\sim \frac 1\rho  t^{\rho(1-\frac 1t)}$ as $t\to\infty$.
%\end{rem}

Next results concern oscillation of solutions of  equation \eqref{e:oscil_eq}.
\begin{theorem} \label{t:zeros_de}
 Let conditions of Theorem \ref{t:sufficient_cond_interpol} be satisfied. Then there exists an analytic function $a$ in $\D$ satisfying
\begin{equation*}
 \exists C>0: \ln M(r,a)\le C \tilde \psi \Bigl( \frac 1{1-r} \Bigr), \quad r\in (0,1)
\end{equation*}
such that \eqref{e:oscil_eq} possesses a solution $f$  having zeros  precisely  at the points $z_k$, $k\in \mathbb{N}$.
\end{theorem}

\begin{cor}
 If for some $\rho>0$ a sequence $(z_k)$ satisfies the condition
\begin{equation*}
 \exists C>0: N_{z_k} \Bigl(  \frac{1-|z_k|}{2}\Bigr) \le C \Bigl( \frac{1}{1-|z_k|}\Bigr)^\rho,
\end{equation*}
then there exists a function $a$ analytic in $\D$ satisfying $\ln M(r,a)=O((1-r)^{-\rho})$, $r\in (0,1)$ such that
possesses a solution $f$  having zeros  precisely  at the points $z_k$, $k\in \mathbb{N}$.
\end{cor}

The following theorem is based on  an example due to J. Gr\"ohn and J. Heittokangas \cite{Gr_Heit}.
\begin{theorem} \label{t:zeors_sharp} For arbitrary $\rho>0$ there exists a sequence of distinct numbers $\{z_n\}$ in $\D$ with the following properties:
\begin{itemize}
 \item [i)] $N_{z_k} \Bigl(  \frac{1-|z_k|}{2}\Bigr) \le C \Bigl( \frac{1}{1-|z_k|}\Bigr)^\rho$, $k\in \mathbb{N}$;
\item [ii)] $(z_k)$ cannot be the zero sequence of a solution of \eqref{e:oscil_eq}, where $\ln M(r,a)=O((1-r)^{-\rho+\varepsilon_0})$ for  any $\varepsilon_0>0$.
\end{itemize}

\end{theorem}

\section{Preliminaries}
\subsection{Some auxiliary results}
For  $s=[\rho]+1$, where $\rho=\rho^*[\psi]$, we consider a canonical product of the form
\begin{equation}\label{e:can_prod}
P(z)=P(z, Z,s)=\prod_{n=1}^\infty E\Bigl( A_n(z),s\Bigr),
\end{equation}
where  $E(w,0)=1-w$,  $E(w,s)=(1-w)\exp \{w+ w^2/2 +\dots+ w^s/s\}$, $ s\in \N$, and $A_n(z)=\dfrac{1-|z_n|^2}{1-\bar z_n z}$. This product is an analytic function in $\mathbb{D}$ with the zero sequence  $Z=(z_n)$ provided  $\sum\limits_{z_n\in Z} (1-|z_n|)^{s+1}<\infty$.

The following two lemmas  play a key role in the proofs of the theorems.
The first one is a generalization of Linden's lemma from \cite{L_rep}, which  gives an upper estimate of the canonical product
via $n_z\bigl( \frac{1-|z|}2\bigr)$.

\begin{lem}\label{l:lin_can} Let a sequence  $Z=(z_n)$ in the unit disc be such that
%\[ \nu(re^{i\varphi}) \le  \psi \Bl \frac 1{1-r} \Br, \quad r\in[0,1),
%\]
 \[ n_z\Bigl( \frac{1-|z|}2\Bigr)\le  \psi \Bl \frac 1{1-|z|} \Br, \quad z\in \mathbb{D}, \]
where the function  $\psi$ satisfies \eqref{e:polya}. Then for $s>\rho^*[\psi]$ the canonical product
  $P(z)=P(z,Z,s)$  of the form  \eqref{e:can_prod} admits the estimate
\begin{equation}\label{e:lem1}
\log |P(z)|\le 2^{s+2} \sum_{n=1}^\infty \Bm A_n(z)\Bm^{s+1} \le  {C} \tilde \psi\Bl \frac 1{1-|z|} \Br , \quad z\in \D,
\end{equation}
for some constant $C>0$.
\end{lem}

\begin{rem} \label{r:subprod} Note that any canonical subproduct of $P$ satisfies \eqref{e:lem1} with the same constant $C$. In fact, let $Z_1\subset Z$, and $P_1(z)=P(z, Z_1, s)$. Then
\begin{equation*}%\label{e:lem1}
\log |P_1(z)|\le 2^{s+2} \sum_{z_n\in Z_1} \Bm A_n(z)\Bm^{s+1}\le 2^{s+2} \sum_{n=1}^\infty  \Bm A_n(z) \Bm^{s+1}  \le  {C} \tilde \psi\Bl \frac 1{1-|z|} \Br .
\end{equation*}
\end{rem}
We write $B_k(z)=\dfrac {P(z)}{E(A_k(z),s)}$.

%The following lemma plays a key role in the proof of Theorem \ref{t:sufficient_cond_interpol}.
\begin{lem}  \label{l:index}
For an arbitrary  $\delta\in (0,1)$, any sequence $Z$ in $\D$ satisfying $\sum_{z_k\in Z} (1-|z_k|)^{s+1}< \infty$, $s\in \Z_+$  there exists a positive constant  $C(\delta, s)$
$$| \ln |B_k(z_k)| + N_{z_k} (\delta(1-|z_k|))|\le C(\delta, s)\sum_{n=1}^\infty |A_n(z_k)|^{s+1},\quad k\to+\infty.$$
\end{lem}

The next proposition compares some conditions frequently used in interpolation problems.
\begin{prop} \label{p:1}
Given  a function $\psi \in \mathcal{R}$ for $$ \exists C>0\ \forall z
\in \mathbb{D}:\  N_z\Bl\frac {1-|z|}2\Br\leq  C\psi\Bl \frac 1{1-|z|}\Br
$$
it is necessary and sufficient that
\begin{equation}\label{e:nu_est_rho}
 \exists C>0\ \forall z
\in \mathbb{D}:     n_{z} \Bl \frac{1-|z|}2\Br \le  C\psi\Bl \frac 1{1-|z|}\Br,
\end{equation}
and
\begin{equation}\label{e:ln_prime}
 \forall n\in \mathbb{N}:\;    |\ln \bigl((1-|z_n|)|P'(z_n)|\bigr)| \le  C\psi\Bl \frac 1{1-|z_n|}\Br,
 \end{equation}
where $P$ is the canonical product defined by (\ref{e:can_prod}), $s=[\rho]+1$, where $\rho$ is Polya's order of $\psi$.
\end{prop}

%Note that conditions of the form \eqref{e:ln_prime} appeared earlier as  sufficient conditions on $Z$ to be an interpolation set.

\subsection{Proofs of the lemmas}
\begin{proof}[Proof of Lemma \ref{l:lin_can}] The proof repeats, in general, the original Linden's one (\cite{L_rep}), therefore we only sketch it, emphasizing distinctions. Without loss of generality we may assume that $\frac 12\le |z_n|<1$, $n\in \mathbb{N}$. The first inequality \eqref{e:lem1} is the assertion of Tsuji's theorem  \cite{Ts}. To prove the second one we denote
 \[\square(\rr\vfi)=\Bigl\{\rho e^{i\theta}:  r\le \rho < \frac {1+r}2 ,
|\theta  -\vfi|\le  \frac{1-r}4\Bigr\}, \] and  $\nu(\rr\vfi)$  being
the number of members of the sequences  $Z=(z_n)$ in $\square(\rr\vfi)$;
$S_{h,k}(\varphi)=\square\Bl(1-2^{-k})e^{i(\varphi+\pi 2^{-k}(2h+1))}\Br$.
Note that the conditions  $n_z\bigl( \frac{1-|z|}2\bigr)=O(\psi(\frac 1{1-|z|}))$ and $\nu(z)=O(\psi(\frac 1{1-|z|}))$  are equivalent as $|z|\uparrow 1$.

Then, for  $z_n\in S_{h,k}(\varphi)$, $z=re^{i\varphi}$ we have (see \cite[p.24]{L_rep})
\begin{equation*}
    \Bl \frac {1-|z_n|^2}{|1-z\bar z_n|}\Br^{s+1} \le \frac {1}{((1-r+ r2^{-k-1})^2 +h^22^{2-2k})^{\frac {s+1}2}}.
\end{equation*}
Thus, similar to  that as  one deduces formula  (18) from \cite{L_rep}, we obtain
\begin{gather}
\nonumber
\sum_{2^{-k-1}<1-|z_n|\le 2^{-k}}      \Bl \frac {1-|z_n|^2}{|1-z\bar z_n|}\Br^{s+1} \le  \\ \le \sum_{h=0}^{2^{k-1}-1}\frac {2 \psi(2^k)}{((1-r+ r2^{-k-1})^2 +h^22^{2-2k})^{\frac {s+1}2}} \le
\frac{2\psi(2^k)(8+B(\frac 12, \frac s2))}{2^{ks} (1-r+r2^{-k-1})^s},
\label{e:annul_est}
    \end{gather}
where $B(x,y)$ is the Beta-function.

Let  $r=|z|\in [1-2^{-\nu}, 1-2^{-\nu-1})$.
It follows from \eqref{e:annul_est} and \eqref{e:psi_con} that
\begin{gather*}\nonumber
    \sum_{k=\nu+1}^\infty  \sum_{2^{-k-1}<1-|z_n|\le 2^{-k}} \Bl \frac {1-|z_n|^2}{|1-z\bar z_n|}\Br^{s+1} \le  C(s)\sum_{k=\nu+1}^\infty   \frac{\psi(2^k)}{2^{ks}(1-r)^s} \le \\
    \le \frac {C(s)\psi(2^{\nu+1})}{(1-r)^s2^{(\nu+1)s}} \sum_{k=\nu+1}^\infty 2^{(s-\rho)(k-\nu-1)} \le {C(s,\rho)}\psi\Bl \frac1{1-r}\Br.\label{e:tail_est}
\end{gather*}
Further,  \eqref{e:annul_est} implies
\begin{gather*}
    \sum_{k=1}^\nu   \sum_{2^{-k-1}<1-|z_n|\le 2^{-k}}\Bl \frac {1-|z_n|^2}{|1-z\bar z_n|}\Br^{s+1} \le C(s)  \sum_{k=1}^\nu  \psi(2^k)\le \\ \le C(s, \rho) \sum_{k=1}^\nu \int_{1-2^{1-k}}^{1-2^{-k}} \frac{\psi (\frac 1{1-t})}{1-t} \, dt %\le \\ %\le C \int_0^{1-2^{-\nu}} \frac {\psi (\frac 1{1-t})}{1-t} \, dt =
\le     C \tilde\psi(2^\nu) \le C \tilde \psi\Bl \frac 1{1-r}\Br.
\end{gather*}
It is well-known that $\psi(x)=O(\tilde\psi(2x))$ $(x\to+\infty)$, so the assertion of the lemma follows from the two latter estimates.
\end{proof}

\begin{proof}[Proof of Lemma \ref{l:index}]
Without loss of generality, we assume that $Z$ is an infinite sequence, $|z_n|\ge \frac 12$, $n\in \mathbb{N}$, and $s\in \N$.
We denote  $\delta(1-|z_k|)=\eta_k$, and note that
$$ N_{z_k} (\eta_k)=\sum_{0< |z_n-z_k|\le \eta_k}  \ln \frac{\eta_k}{|z_n-z_k|}=\int_0^{\eta_k} \frac{n_{z_k}(x)-1}x dx.$$
Then
\begin{gather}\nonumber
\ln |B_k(z_k)|+ N_{z_k} (\eta_k) =\\ = \sum_{n\ne k} \Bigl( \ln \Bm \frac{(z_k-z_n)\bar z_n}{1-\bar z_n z_k}\Bm+ \Re \sum_{j=1}^s \frac 1j (A_n(z_k))^j\Bigr)  \nonumber  +\sum_{0< |z_n-z_k|\le \eta_k} \ln \frac{\eta_k}{|z_n-z_k|}= \\ =\sum_{0< |z_n-z_k|\le \eta_k}
\Bigl( \ln  \frac{\eta_k |z_n|}{|1-\bar z_n z_k|}+  \Re \sum_{j=1}^s \frac 1j (A_n(z_k))^j\Bigr) + \nonumber\\ +  \sum_{|z_n-z_k|> \eta_k}
\ln |E(A_n(z_k),s)|.
\label{e:error_term}
    \end{gather}
It is easy to see (\cite[p.~528]{Illin}) that $|1-\bar z_n z_k|\le (2+\delta)(1-|z_k|)$ for $|z_n-z_k|\le \eta_k$. Taking into account that
$|A_n(z)|\le 2$, $z\in \mathbb{D}$, we obtain for $|z_n-z_k|\le \eta_k$
\begin{gather*}\label{e:A_n_est}
    C'(\delta)\le |A_n(z_k)|\le C''(\delta), \quad
    \frac \delta {2(2+\delta)} \le \frac {\eta_k |z_n|} {|1-\bar z_n z_k|} \le  \delta .
\end{gather*}
Therefore for the first sum in the right-hand side of \eqref{e:error_term} we get
\begin{gather}\nonumber
\biggl|    \sum_{0< |z_n-z_k|\le \eta_k}
\Bigl( \ln  \frac{\eta_k |z_n|}{|1-\bar z_n z_k|}+ \Re \sum_{j=1}^s \frac 1j (A_n(z_k))^j\Bigr)  \biggr|\le \\  \le \sum_{0< |z_n-z_k|\le \eta_k}
C(s, \delta) \le C(s, \delta) \sum_{0< |z_n-z_k|\le \eta_k} |A_n(z_k)|^{s+1}.\label{e:first_sum}
\end{gather}
%\end{proof}
Tsuji   (see \cite[p.8]{Ts}) proved that for an appropriate branch of the logarithm
\begin{equation}\label{e:in8}
\sum_{|A_n(z)|< \frac12}|\ln E(A_n(z),s)|\le 2 \sum_{|A_n(z)|<
\frac12} |A_n(z)|^{s+1}.
\end{equation}

It can be checked  that
for the pseudohyperbolic disc $\mathcal{D}(z,s)=\Bigl\{
\zeta: \sigma(z, \zeta)<s\Bigr\}$ (\cite[I.1]{Gar}) the inclusion
\begin{equation*}%\label{e:hyper}
\mathcal{D}\Bigl(z,\frac{\delta}{2+\delta}\Bigr)\subset U(z, (1-|z|)\delta)
\end{equation*} holds (see \cite[p.529]{Illin} for details).
%It is sufficient to show that $|z^*-z|+\rho_z(s)\le h(1-|z|)$ for
%$s\le h/(2+h)$. We have $(|z|=r)$
%\begin{gather*}
%|z^*-z|+\rho_z(s)=\frac{(1-r^2)(rs^2+s)}{1-s^2r^2}\le
%\frac{2(1-r)s}{1-s}.
%\end{gather*}
%Thus, we arrive to the inequality $2s\le h(1-s)$, which is
%equivalent to $-1\le s\le \frac h{2+h}$. Inclusion (\ref{e:hyper})
%is proved. Therefore,
Thus, for   $z_n\not \in U(z_k, \eta_k)$  we have
$$ 0\ge \ln \Bm \frac {\bar z_n (z_n -z_k)}{1-\bar z_n z_k}\Bm \ge \ln \frac{\delta}{2(2+\delta)}.$$
Hence
\begin{gather*}\nonumber
\sum_{\begin{substack}{|A_n(z_k)|\ge \frac12 \\ |z_n-z_k|>\eta_k} \end{substack}}
\ln |E(A_n(z),s)|  = \sum_{\begin{substack}{|A_n(z_k)|\ge \frac12 \\ |z_n-z_k|>\eta_k} \end{substack}} \Bigl( \ln \Bm \frac{(z_k-z_n)\bar z_n}{1-\bar z_n z_k}\Bm+ \Re \sum_{j=1}^s \frac 1j (A_n(z_k))^j \Bigr)\ge  \\ \ge - \sum_{\begin{substack}{|A_n(z_k)|\ge \frac12 \\ |z_n-z_k|>\eta_k} \end{substack}} \Bigl( \ln \frac{2(2+\delta)}\delta + \sum_{j=1}^s \frac{2^j}j\Bigr)\ge - C(\delta,s) \sum_{\begin{substack}{|A_n(z_k)|\ge \frac12 \\ |z_n-z_k|>\eta_k} \end{substack}} |A_n(z_k)|^{s+1}.
\label{e:in8'}
\end{gather*}
%Далі за теоремою Цудзі маємо
%$$ \sum_{n=1}^\infty $$
Relations \eqref{e:error_term}--\eqref{e:in8} and the last inequality give the assertion of the lemma.
\end{proof}

\section{Proofs of interpolation theorems and Proposition \ref{p:1}}
\begin{proof}[Proof of Theorem \ref{t:sufficient_cond_interpol}]
First of all, we note that the estimate
$$\Bl n_z\Bl\frac \delta\alpha(1-|z|) \Br-1\Br^+ \ln \alpha \le \int_{\frac \delta\alpha(1-|z|)}^{\delta(1-|z|)} \frac{(n_z(x)-1)^+}{x}\, dx\le N_z(\delta(1-|z|)), $$
where $ 0<\alpha\delta<1<\alpha $, yields
\begin{equation}\label{e:nu_est}
    \max _{\theta} n_{re^{i\theta}} (\delta (1-r))\le {C } \psi\Bl {\frac1{1-r}}\Br.
\end{equation}

%We continue the proof of the theorem.
It follows from the estimate \eqref{e:nu_est} and Lemma  \ref{l:lin_can} that
$$\sum_{n=1}^\infty |A_n(z)|^{s+1}\le C(s) \tilde \psi\Bl \frac 1{1-|z|}\Br, \quad z\in \mathbb{D}.$$
From this estimate and Lemma  \ref{l:index} we deduce
\begin{equation}\label{e:error_est_ro}
|\ln |B_k(z_k)|| \le C(\delta,s) \tilde \psi \Bl \frac 1{1-|z_k|}\Br, \quad k\to+\infty.
\end{equation}

Consider the interpolation function
\begin{equation*}\label{e:inter_func}
    f(z)=\sum_{n=1}^\infty \frac {b_n}{z-z_n} \frac{P(z)}{P'(z_n)}\Bl \frac{1-|z_n|^2}{1-\bar z_n z}\Br^{s_n-1},
\end{equation*}
where an increasing sequence of natural numbers  $(s_n)$ will be specified below. It is not hard to check that \eqref{e:inter_prob} holds.

Moreover, taking into account Remark \ref{r:subprod} and the inequality $|A_n(z)|\le 2$,   we have the following estimates
\begin{gather}\nonumber
\Bm \frac {1-\bar z_n z}{\bar z_n} \frac {P(z)}{z-z_n}\Bm =\Bm B_n(z) \exp\Bigl\{ \sum_{j=1}^s \frac 1j (A_n(z))^j \Bigr\}\Bm \le \\
\le \exp \Bigl\{C(s) \sum_{n=1}^\infty |A_n(z)|^{s+1}\Bigr\}\le \exp\Bigl\{C \tilde \psi \Bl\frac {1}{1-|z|}\Br\Bigr\},
\label{e:growth_interpol_est}\\
\frac{P'(z_n)(1-|z_n|^2)}{\bar z_n}= -B_n(z_n) \exp\Bigl\{1+\frac 12+\dots +\frac 1s\Bigr\}. \label{e:p'est}
\end{gather}
Therefore, using our assumption on $(b_n)$,  \eqref{e:growth_interpol_est}, and \eqref{e:p'est}, we deduce
\begin{gather} \nonumber
    |f(z)|=\biggl| \sum_{n=1}^\infty b_n \frac{P(z)(1-\bar z_n z)}{\bar z_n (z-z_n)} \frac{\bar z_n}{(1-|z_n|^2)P'(z_n)} \Bl \frac{1-|z_n|^2}{1-\bar z_n z}\Br^{s_n}\biggr| \le \\
    \le \sum_{n=1}^\infty \exp\Bigl\{ C\tilde \psi \Bl \frac {1}{1-|z_n|}\Br \Bigr\} \exp\Bigl\{  {C}\tilde \psi \Bl \frac1{1-|z|}\Bigr\} \frac 1{|B_n(z_n)|} \Bl \frac{1-|z_n|^2}{|1-\bar z_n z|}\Br^{s_n} \le \nonumber
        \\ \le \exp\Bigl\{ C \tilde \psi\Bl\frac {1}{1-|z|}\Br \Bigr\} \sum_{n=1}^\infty \exp\Bigl\{ C'' \tilde \psi\Bl\frac {1}{1-|z_n|}\Br \Bigr\}  \Bl \frac{1-|z_n|^2}{|1-\bar z_n z|}\Br^{s_n}. \label{e:|f|_est}
\end{gather}

If  $\psi$ is nondecreasing then $\tilde \psi$ is convex with respect to the logarithm. The condition $\psi(t)\to+\infty$ yields $\ln x=o(\tilde \psi(x))$ $(x\to+\infty)$.
By Clunie-K\"ovari's theorem \cite{ClKo}, given a positive constant $C_0$ there exists an entire function  $\Phi(z)=\sum_{n=0}^\infty \vfi_n z^n$ such that $\ln M(t, \Phi)=(C_0+o(1)) \tilde \psi(C_0 t)$ $(t\to\infty)$. Then there is  $t_0>0$ such that the following estimates of the maximal term $\mu(t,\Phi)=\max\{|\vfi_n|t^n : n\in \mathbb{Z}_+\}$ are valid:
\begin{gather*}\label{e:max_term_est_up}
    \mu(t, \Phi)\le \exp\{ 2 \tilde \psi (C_0t)\} ,\quad \mu(t, \Phi)\ge \exp\Bigl\{ \frac 14 \tilde \psi \Bigl(\frac{C_0t}2\Bigr)\Bigr\}  \quad t\ge t_0.\\
%     , \quad t\ge t_0.  \label{e:max_term_est_low}
\end{gather*}
Let  $\hat \Phi (z)=\sum_{n=0}^\infty  \hat\vfi_n z^n$ be Newton's majorant for the function $\Phi(z)$ (see \cite[Chap. IX, \S 68]{Val}). Then $\mu(r, \Phi)=\mu(r, \hat \Phi)$, and the sequence  $\varkappa_n=\hat \vfi_{n-1} /\hat \vfi_n $, $\varkappa_0=0$ is unbounded and increasing. We choose the sequence $(s_n)$   such that  $\varkappa _{s_n} \le \frac 1{1-|z_n|} < \varkappa _{s_{n+1}} $.
Then
\begin{equation}\label{e:max_term_n}
    \mu\Bigl(\frac 1{1-|z_n|},\Phi\Bigr)= \hat \varphi_{s_n} \frac 1{(1-|z_n|)^{s_n}}, \quad  \mu(t,\Phi)\ge  \hat \varphi_{s_n} t^{s_n}.
\end{equation}
Using the obtained inequalities and choosing  $C_0>\max\{ 8C'', 2\}$, we deduce
\begin{gather}\allowdisplaybreaks \nonumber
 \sum_{n=1}^\infty \exp\Bigl\{ C'' \tilde \psi\Bl\frac {1}{1-|z_n|}\Br \Bigr\}  \Bl \frac{1-|z_n|^2}{|1-\bar z_n z|}\Br^{s_n} \le \\
\le \sum_{n=1}^\infty \exp\Bigl\{ C'' \tilde \psi\Bl\frac {1}{1-|z_n|}\Br \Bigr\}  \frac {\hat \vfi_{s_n}\Bl \frac2{1-|z|}\Br ^{s_n}}{
\hat \vfi_{s_n} \Bl \frac1{1-|z_n|}\Br^{s_n}} \le \nonumber \\ \allowdisplaybreaks
\le \sum_{n=1}^\infty \exp\Bigl\{ C'' \tilde \psi\Bl\frac {1}{1-|z_n|}\Br \Bigr\}  \frac {\mu \Bl \frac2{1-|z|}, \hat \Phi\Br}{
\mu \Bl \frac1{1-|z_n|}, \hat \Phi \Br} \le \nonumber \\
\le \sum_{n=1}^\infty \exp\Bigl\{ C'' \tilde \psi\Bl\frac {1}{1-|z_n|}\Br + 2C_0\tilde\psi \Bl\frac{2C_0}{1-|z|} \Br - \frac {C_0}4\tilde\psi \Bl\frac{C_0}{2(1-|z_n|)}\Bigr\}  \le \nonumber \\
\le \exp\Bigl\{ 2C_0\tilde\psi \Bl\frac{2C_0}{1-|z|} \Br \Bigr\} \sum_{n=1}^\infty \exp\Bigl\{ - \frac {C_0}8\tilde\psi \Bl\frac{C_0}{2(1-|z_n|)}\Br\Bigr\}  \le \nonumber \\ \ \le \exp\Bigl\{ 2C_0\tilde\psi \Bl\frac{2C_0}{1-|z|} \Br \Bigr\} \sum_{n=1}^\infty \exp\Bigl\{ -(s+1)\ln \frac 1{1-|z_n|} \Bigr\} \le \nonumber \\ \le   C(s) \exp\Bigl\{ 2C_0\tilde\psi \Bl\frac{2C_0}{1-|z|}\Br \Bigr\}.
    \label{e:kovar_stuff}
\end{gather}

Substituting estimate  \eqref{e:kovar_stuff} in \eqref{e:|f|_est},
we get  \eqref{e:growth_con_psi}.
\end{proof}

\begin{proof}[Proof of Theorem \ref{t:interpol_criter}]
The implications $(i)\Rightarrow(ii)$  and $(ii)\Rightarrow(i)$ follow from the theorem  of B.\ V.\ Vynnytskyi, I.\ B.\ Sheparovych (\cite{VSh04}),   and Theorem \ref{t:sufficient_cond_interpol}, respectively.

We then show that (ii) and (iii) are equivalent. \label{p:2iff3}
%азначимо, що якщо виконується умова
%$$ (\exists\delta\in (0;1)) (\exists c_{6})(\forall n): \prod\limits_{|\lambda_
%i-\lambda _{n}|\leqslant \delta (1-|\lambda _{n}|)}\left|
%\frac{\lambda_ n-\lambda _{i}}{1-\overline{\lambda _{i}}\lambda
%_{n}}\right| \geqslant \exp \left( -\frac{c_{6}}{(1-|\lambda
%_{n}|)^{\rho+\varepsilon}}\right),\eqno(35)
%$$
%то виконується (10).

As it was  proved,  for $|z-z_n|\le \delta(1-|z|)$ we have
 $1-|z|\le |1-\bar z_n z| \le (2+\delta)(1-|z|)$, i.e. $ 0\le \ln \frac{|1-\bar z_n z|}{1-|z|} \le \ln(2+\delta)$.
 Hence,
 \begin{multline}\label{e:compar_concentr}
    0\le \sum\limits_{0<|z_k-z_n|\leqslant\delta(1-|z_k|) } \biggl(\ln \frac 1{\sigma(z_n,z_k)} -\ln \frac{1-|z_k|}{|z_n-z_k|}\biggr)=\\ =\sum\limits_{|z_k-z_n|\leqslant\delta(1-|z_k|) }\ln\frac{|1-\bar z_nz_k|}{1-|z_k|}
\leqslant  N_{z_k}(\alpha\delta(1-|z_k|))\ln (2+\delta) \le \\ \le  N_{z_k}(\alpha\delta(1-|z_k|)) \ln (2+\delta)\ln \alpha.
 \end{multline}
Further,
\begin{gather*}
 \sum\limits_{0<|z_k-z_n|\leqslant\delta(1-|z_k|) } \ln \frac {1-|z_k|}{|z_k-z_n|} =
\int_0^{\delta(1-|z_k|)}\ln\frac
{1-|z_k|}\tau d(n_{z_k}(\tau)-1)
=\\ =\ln\frac1\delta\cdot(n_{z_k}(\delta(1-|z_k|))-1)+ N_{z_k}(\delta(1-|z_k|)).
\end{gather*}
Therefore
$$N_{z_k}(\delta(1-|z_k|)) \le  \sum\limits_{0<|z_k-z_n|\leqslant\delta(1-|z_k|) } \ln \frac {1-|z_k|}{|z_k-z_n|} \le  $$ $$ \le N_{z_k}(\alpha \delta(1-|z_k|))\bigl(1+ \ln \alpha \ln \frac 1\delta\bigr).$$
The latter inequality together with \eqref{e:compar_concentr} proves the equivalence between (ii) and~(iii).
%Справді, оскільки
%$\frac1{|1-\overline{\lambda}_i\lambda_n|}\leqslant\frac1{1-|\lambda_n|}$,
%то
%$$
%\sum\limits_{\mathrel{\mathop{i\ne n,}\limits_{{\scriptstyle
%|\lambda_i-\lambda_n|\leqslant\delta(1-|\lambda_n|) }}}}
%\ln\left|\frac{\lambda_i-\lambda_n}{1-\overline{\lambda}_i\lambda_n}\right|
%\leqslant \sum\limits_{\mathrel{\mathop{i\ne
%n,}\limits_{{\scriptstyle
%|\lambda_i-\lambda_n|\leqslant\delta(1-|\lambda_n|) }}}}
%\ln\frac{|\lambda_i-\lambda_n|}{1-|\lambda_n|}=
%$$
%$$
%=\int_0^{\delta(1-|\lambda_n|)}\ln\frac
%x{1-|\lambda_n|}d(n_{\lambda_n}(x)-1)
%=\ln\delta\cdot(n_{\lambda_n}(\delta(1-|\lambda_n|))-1)-
%$$
%$$
%-\int_0^{\delta(1-|\lambda_n|)}\frac{n_{\lambda_n}(x)-1}xdx\leqslant
%-N_{\lambda_n}(\delta(1-|\lambda_n|)).
%$$
% Тому справедливою є
% \begin{theorem} \label{t:sufficient_cond_interpol} Нехай $(z_n)$ --- послідовність різних комплексних чисел в $\mathbb{D}$, $\rho\in(0, +\infty)$.
%Припустимо,  для деякої неспадної функції  $\psi\colon
%[1,+\infty)\to \mathbb{R}_+ $, яка задовольняє \eqref{e:psi_con},
%виконується (35). Тоді для довільної послідовності  $(b_n)$ з
%властивістю (11)
%  існує аналітична в $\mathbb{D}$ функція $f$, яка задовольняє
%  умови (1) і (13).
% \end{theorem}
\end{proof}

\begin{proof}[Proof of Proposition \ref{p:1}] In fact, the necessity of  (\ref{e:nu_est_rho}) has been already established  in the proof of Theorem~\ref{t:sufficient_cond_interpol}.
Necessity of (\ref{e:ln_prime}) and sufficiency  follow from Lemmas \ref{l:lin_can}, \ref{l:index}, formula \eqref{e:p'est},  and Remark \ref{r:N_z}.
\end{proof}

\section{Proofs of the oscillation theorems}

\begin{proof}[Proof of Theorem \ref{t:zeros_de}]
 Let $f(z)=P(z)e^{g(z)}$, be analytic in $\D$ where $P$ is  the canonical product defined by \eqref{e:can_prod} with the zeros sequence $Z=(z_k)$. We can rewrite  \eqref{e:oscil_eq}  as
\begin{equation}
 \label{e:osc_p_g}
P''+2 P'g' +(g'^2 +g'' +a)P=0,
\end{equation}
and, consequently
\begin{equation}
 \label{e:g_interpol}
g'(z_k)=-\frac{P''(z_k)}{2P'(z_k)}=:b_k, \quad k\in \mathbb{N}.
\end{equation}
Therefore, in order to find  a solution of \eqref{e:oscil_eq} with the zero sequence $Z$ we have to find an analytic function $h=g'$ solving the interpolation problem $h(z_k)=b_k$, $k\in \N$.
Using Cauchy's integral theorem  and Lemma \ref{l:lin_can} we deduce
\begin{equation*}
 |P''(z_k)|\le \frac{8}{(1-|z_k|)^2} \max _{|z|=\frac{1+|z_k|}{2}} |P(z)|\le \frac{8}{(1-|z_k|)^2} e^{C \tilde \psi (\frac{ 2}{1-|z_k|})}.
\end{equation*}
On the other hand, \eqref{e:p'est} and \eqref{e:error_est_ro} imply  (cf. \eqref{e:ln_prime}) that
$$ \frac{1}{|P'(z_k)|} \le (1-|z_k|) e^{C\tilde \psi (\frac{ 1}{1-|z_k|})}.$$
Hence
\begin{gather*}
 |b_k|=\Bigl| \frac{P''(z_k)}{2P'(z_k)}\Bigr| \le \frac{4}{1-|z_k|} e^{C \tilde \psi (\frac{ 1}{1-|z_k|})}= \\ =
e^{C \tilde \psi (\frac{ 2}{1-|z_k|})+ \ln \frac{4}{1-|z_k|}}\le e^{C \tilde \psi (\frac{ 1}{1-|z_k|})}, \quad k\in \N,
\end{gather*}
because $\tilde\psi (t)/\ln t\to+\infty$ $(t\to+\infty)$. Since the assumptions of Theorem \ref{t:sufficient_cond_interpol} are satisfied there exists a function $h$  analytic in $\D$ such that $h(z_k)=b_k$ and $\ln M(r,h)\le C \tilde \psi (\frac{1}{1-r})$,
$r\uparrow1$, i.e. $\ln M(r,g')\le C \tilde \psi (\frac{1}{1-r})$,
$r\uparrow1$.

Then, applying Cauchy's theorem once more, we get that
$$ M(r,g'')\le \frac{2 }{1-r} M\Bigl( \frac{1+r}{2}, g'\Bigr) \le e^{C \tilde \psi (\frac 1{1-r})}, \quad r \uparrow1.$$
From \eqref{e:osc_p_g} we obtain
$$ |a(z)| \le \Bigl| \frac{P''(z)}{P(z)}\Bigr|+ 2|g'(z)| \Bigl| \frac{P'(z)}{P(z)}\Bigr|+ |g'(z)|^2+ |g''(z)|.$$
It follows from results of \cite{ChyGuHe} or \cite{CHR1} that for any $\delta>0$ there exists a set $E_
\delta \subset [0,1)$ such that
\begin{equation*}
 \max \Bigl\{  \frac{|P''(z)|}{|P(z)|}, \frac{|P'(z)|}{|P(z)|} \Bigr\} \le \frac{1}{(1-|z|)^q}, \quad |z|\in [0,1)\setminus E_\delta,
\end{equation*}
where $q\in (0,+\infty)$,  and  $m_1(E_\delta \cap [r,1))\le \delta(1-r)$ as $r \uparrow1$.
Thus, \begin{equation}\label{e:a_est_excep}
       |a(z)| \le e^{\tilde C\tilde\psi (\frac{1}{1-|z|})} , \quad |z|\in [0,1)\setminus E.
      \end{equation}
Since $M(r,a)$ increases,  condition \eqref{e:polya} and Lemma 4.1 from \cite{CHR1} imply that  inequality \eqref{e:a_est_excep} holds for all $z\in \D$ for an appropriate choice of $\tilde C$.
\end{proof}

\begin{proof}[Proof of Theorem \ref{t:zeors_sharp}]
 Let $\rho>0$ be given. Let $\varepsilon_n= \frac12 e^{-2^{n\rho}}$, $n\in \N$. Let $(z_n)$ be the sequence defined  by $$z_{2n-1}=1-2^{-n}, \quad z_{2n}=1-2^{-n}+\varepsilon_n.$$
Then for $m\in \{2n-1, 2n\}$ we have as $ n\to\infty$
\begin{gather*}
 N_{z_m} \Bigl( \frac{1-|z_m|}{2}\Bigr)=\sum\limits_{0<|z_k-z_m|\leqslant\frac12(1-|z_m|) } \ln \frac {1-|z_k|}{2|z_k-z_m|} =\\ =\ln \frac{2^{-n} +O(\varepsilon_n)}{2\varepsilon _n} +O\Bigl(\ln \frac{2^{-n}+O(\varepsilon_n)}{2^{-n}+O(\varepsilon_n)}\Bigr)
= 2^{n\rho} +O(n)\sim\Bigl(\frac1{1-|z_m|}\Bigr)^{\rho}.\end{gather*}
Thus, assertion i) is proved.

To prove assertion ii) we assume on the contrary that there exists a solution $f=Be^g$ of \eqref{e:oscil_eq} having the zero sequence $(z_n)$, where $B$ is the Blaschke product, and such that
\begin{equation}\label{e:assum_a}
\ln M(r,a)\le C (1-r)^{-\rho+\varepsilon_0},
\quad r\in [0,1), \varepsilon_0>0.                                                                                                                                                                                                                    \end{equation}
%A routine calculation yields
Repeating the arguments from the proof of Theorem 5 \cite{Heit_cmft} (the only difference is that we have smaller $\varepsilon_n$), one can show that

%\begin{gather}\nonumber
 %P'(z_m)=-\frac{\bar z_m}{1-|z_m|^2} e^{ \sum_{j=1}^s \frac 1j} \prod_{n\ne m} E(A_n(z_m),s),\\
%P''(z_m)=-\frac{\bar z_m}{1-|z_m|^2} e^{ \sum_{j=1}^s \frac 1j} \prod_{n\ne m} E(A_n(z_m),s),\\
%-\frac{P''(z_m)}{2P'(z_m)}=\sum_{n\ne m} \frac{1}{z_n-z_m} ( A_n(z_m))^{s+1} -\frac{(s+1)\bar z_m}{1-|z_m|^2}
%\end{gather}
%We write
%\begin{gather}\nonumber
 %-g'(z_{2m})= \frac{P''(z_{2m})}{2P'(z_{2m})}=\frac{1}{z_{2m} -z_{2m-1}} \Bigl( \frac{1-|z_{2m-1}|^2}{1-\bar{z}_{2m-1} z_{2m}}\Bigr)^{s+1}+ \\ +\sum_{j\ne 2m, 2m-1} \frac{1}{z_{2m} -z_j}\Bigl( \frac{1-|z_{j}|^2}{1-\bar{z}_{j} z_{2m}}\Bigr)^{s+1} + \frac{(s+1)\bar{z}_{2m}}{1-|z_{2m}|^2}=: I_1+I_2+I_3.
%\end{gather}
%It is easy to see that $|I_1|\ge \frac12 \frac 1{\varepsilon_m}$ as $m\to+\infty$, and $|I_3|\le(s+1)(1-|z_m|)^{-1}$.
%Then
%\begin{gather*}
% |I_2|\le \sum_{j\ne 2m, 2m-1} \frac{1}{|z_{2m} -z_j|}\Bigl( \frac{1-|z_{j}|^2}{|1-\bar{z}_{j} z_{2m}|}\Bigr)^{s+1}\le \sum_{j=1}^{2m-2} \frac{2^{s+2}}{1-|z_j|} + \\ +\sum_{j=2m+1}^\infty \frac{ 2^{s+2}(1-|z_j|)^{s+1}}{(1-|z_{2m}|)^{s+2}} \le 2^{s+2}\sum_{j=1}^{2m-2} 2^{j+1} +
%\frac{2^{s+2}}{(1-|z_{2m}|)^{s+2}}  \sum_{j=2m+1}^\infty 2^{-j(s+1)} \le\\ \le  2^{m+s+2}+ \frac{2^{s+3}}{1-|z_{2m}|} \le \frac{2^{s+4}}{1-|z_{2m}|}.
%\end{gather*}
Therefore $$|g'(z_{2n})|=\Bigl| \frac{B''(z_{2n})}{2B'(z_{2n})}\Bigr|\ge    \frac15 e^{2^{n\rho}}\ge C \exp \frac{1}{(1-|z_{2n}|)^\rho} , \quad n\to+\infty.$$
Hence, \begin{equation}\label{e:g'_growth_low_est}
        \ln M(|z_{2n}|, g')\ge (1-|z_{2n}|)^{-\rho}, \quad n\to+\infty.
       \end{equation}
But \eqref{e:assum_a} implies (see e.g. \cite{HKR}) that
$\ln \ln M(r,f) \le (1-r)^{-\rho+\varepsilon_0}, \quad r\uparrow1.$
Repeating the arguments from the proof of Lemma 2 we get for $R_n=1- 3 \cdot 2^{-n-1}$ and $\delta=\frac 14$
that $( n\to +\infty)$ \begin{gather}\nonumber
    \Re g(R_ne^{i\theta}) \le \ln M(R_n, f)+|\ln |B(R_n e^{i\theta})|| \le \\ \nonumber \le \ln M(R_n, f) +N_{R_n e^{i\theta}}\Bigl(\frac{1-R_n}4\Bigr) + C\Bigl(\frac 14, 1\Bigr) \sum_{k=1}^\infty
 \frac{1-|z_k|^2}{|1-z_kR_n e^{i\theta}|} \le \\ \le\exp \Bigl\{ \Bigl( \frac1{1-R_n}\Bigr)^{\rho-\varepsilon_0}\Bigr\} + O\Bigl( \frac1{1-R_n}\Bigr)^{\max\{\rho, 1\}}\le \exp \Bigl\{ \Bigl( \frac1{1-R_n}\Bigr)^{\rho-\varepsilon_0/2}\Bigr\}. \label{e:max_re_g_est}
     \end{gather}
Since $\Re g     $ is harmonic,  $B(r, \Re g)=\max\{ \Re g(re^{i\theta}): \theta \in [0, 2\pi]\}$ is an increasing function.
It follows from \eqref{e:max_re_g_est} and the relation $1-R_n\asymp 1-R_{n+1}$ that
$$ B(r,\Re  g)\le  \exp\Bigl \{ \Bigl( \frac C{1-r}\Bigr)^{\rho-\varepsilon_0/2}\Bigr\}, \quad r\to\infty.$$
The last estimate and Caratheodory's inequality (\cite[Chap.1, \S 6]{Le}) imply
$$\ln  M(r,g) \le C (1-r)^{-\rho+\varepsilon_0/2}, \quad r\uparrow1.$$
This contradicts to \eqref{e:g'_growth_low_est}. The theorem is proved.
\end{proof}

Faculty of Mechanics and Mathematics,

Ivan Franko National University of Lviv,

 Universytets'ka 1,
79000, Lviv,

Ukraine,

chyzhykov@yahoo.com

\medskip
Institute of Physics, Mathematics and Computer Science,

Drohobych Ivan Franko State Pedagogical  University,

Stryis'ka 3,
Drohobych,

Ukraine,

isheparovych@ukr.net

\medskip

\end{document}